\documentclass[reqno,centertags,12pt]{amsart}

\usepackage{amssymb,amsmath,amsthm,dsfont}

\makeatletter \def\@strippedMR{} \def\@scanforMR#1#2#3\endscan{%
  \ifx#1M\ifx#2R\def\@strippedMR{#3}%
  \else\def\@strippedMR{#1#2#3}%
  \fi\fi} \renewcommand\MR[1]{\relax \ifhmode\unskip\spacefactor3000
  \space\fi \@scanforMR#1\endscan
  MR\MRhref{\@strippedMR}{\@strippedMR}} \makeatother

\newcommand{\R}{\mathbb{R}} \newcommand{\Z}{\mathbb{Z}}
\newcommand{\T}{\mathbb{T}} 
\newcommand{\N}{\mathbb{N}}

\theoremstyle{plain} \newtheorem{theorem}{Theorem}[section]
\newtheorem{lemma}[theorem]{Lemma}
\newtheorem{coro}[theorem]{Corollary}
\newtheorem{prop}[theorem]{Proposition}

\theoremstyle{definition} \newtheorem{definition}[theorem]{Definition}
\theoremstyle{remark} \newtheorem{remark}{Remark}

\newcommand{\eps}{\varepsilon} \newcommand{\lb}{\langle}
\newcommand{\rb}{\rangle}
\newcommand{\ls}{\lesssim}

\begin{document}

\title[Critical NLS on torus]%
{Periodic nonlinear Schr\"odinger equation in critical $H^s(\T^n)$ spaces}
\author[Y.~Wang]{Yuzhao~Wang}


\address{Department of Mathematics and Physics, North China Electric Power University, Beijing 102206, China}
\email{wangyuzhao2008@gmail.com}

\begin{abstract}
In this paper we prove some multi-linear Strichartz estimates for solutions to the linear Schr\"odinger equations on torus $\T^n$. Then we apply it to get some local well-posed results for nonlinear Schr\"odinger equation in critical $H^{s}(\T^n)$ spaces. As by-products, the energy critical global well-posed results and energy subcritical global well-posed results with small initial data are also obtained.
\end{abstract}
\maketitle
\section{Introduction and main result}\label{intro}

\noindent
In this paper we study the initial value problem
\begin{equation}
	\label{eq:nls}
	\begin{cases}
		i u_t - \Delta u  = \pm |u|^{2k} u  \\
		u(0,x) = \phi(x),
	\end{cases}~(x, t)  \in \R \times \T^n,
\end{equation}
where $u$ is a complex-valued function, $k\in \N$ and $\T^n=\R^n/(2\pi \Z)^n$ with $n\ge 1$. Equation \eqref{eq:nls} is called defocusing when the sign is $-$ and focusing with $+$ sign. The solution of \eqref{eq:nls} is invariant under the scaling
\[
u(t,x) \to u_\lambda(t,x)=\lambda^{1/k} u(\lambda^2 t, \lambda x),
\]
with initial data $u_{\lambda,0}=\lambda^{1/k} \phi(\lambda x)$. Let
\begin{align}\label{ind}
s_{n,k}=\frac n2-\frac1k
\end{align}
and it is easy to see that $H^{s_{n,k}}$ is the critical space respect to scaling, see \cite{CA-WE} for more argument on critical spaces. For the equation \eqref{eq:nls} on Euclidean spaces, local well-posed theory and small data scattering theory in critical spaces was established by Cazenave and Weissler \cite{CA-WE}. And Colliander-Kell-Staffilani-Takaoka-Tao \cite{I-team08} extended it to any large data in defocusing case. But there were few well-posed results for periodic case \eqref{eq:nls} in critical spaces until recently.

In \cite{B93a}, Bourgain studied the semilinear Schr\"odinger equations by Strichartz estimates and $X^{s,b}$ spaces, and got the following results:

\begin{theorem}[Bourgain\cite{B93a}]\label{Bourgain} For the initial value problem \eqref{eq:nls}, we have
\begin{itemize}
  \item If $n= 1$ and $k\ge 3$, \eqref{eq:nls} is local well-posed for $\phi \in H^s(\T)$, provided $s> s_{1,k}$.
  \item If $n= 2, 3$ and $k\ge 2$, \eqref{eq:nls} is local well-posed for $\phi \in H^s(\T^n)$, provided $s> s_{n,k}$.
  \item If $n\ge4$ and $k\ge 1$, \eqref{eq:nls} is local well-posed for $\phi \in H^s(\T^n)$, provided $s> s_{n,k}$.
\end{itemize}
Additionally, if $s_{n,k} <1$, then \eqref{eq:nls} is globally well-posed for sufficiently small initial data $\phi \in H^1(\T^n)$.
\end{theorem}

\begin{remark}
Bourgain \cite{B93a} considered more general setting $k\in \R_+$, and here we just list the cases for $k\in \N$, since we only consider such cases in this paper.
\end{remark}

\begin{remark}
This result can be summarized as $H^s$ subcritical local well posedness and $H^1$ subcritical global well-posedness. The global results are relay on the Conservation Laws, for solutions $u$ of \eqref{eq:nls} we have energy conservation
\begin{equation}\label{Energy}
  E(u(t))=\frac12 \int_{\T^n}|\nabla u(t,x)|^2dx+ \frac16
  \int_{\T^n}|u(t,x)|^{2k+2} dx=E(\phi),
\end{equation}
and Mass-conservation
\begin{equation}\label{Mass}
  M(u(t))=\frac12 \int_{\T^n}|u(t,x)|^2dx=M(\phi).
\end{equation}
\end{remark}

Recently, Herr, Tataru and Tzvetkov \cite{HTT11} extend Bourgain's results to critical regularity when $n=3$ and $k=2$, which is energy critical case and they got,
\begin{theorem}[Herr, Tataru and Tzvetkov \cite{HTT11}]\label{HTT11} For the initial value problem
\begin{equation}\label{eq:nls1}
  \begin{cases}
	 i u_t - \Delta u  = \pm |u|^{4} u  \\
     u(0,x) = \phi(x),
  \end{cases}~(x, t)  \in \T^3 \times \R,
\end{equation}
we have local well-posed for $\phi \in H^s(\T^3)$, provided $s\ge 1$. If in addition the $H^1(\T^3)$ norm of the initial data $\phi$ is sufficient small, the solution could be extended to any time.
\end{theorem}

\begin{remark}
By \eqref{ind}, we see that the critical index of \eqref{eq:nls1} is $s_{3,2}=1$, which means that \eqref{eq:nls1} is $H^1$ critical case. This result extend Bourgain's result to the critical end point space $H^1$.
\end{remark}

\begin{remark}
In Bourgain's fundamental paper \cite{B93a}, the sharp Strichartz estimates (for some $p\ge \frac{2(n+2)}{n}$) was deduced under frequency localization, which will cause trouble when we want to sum the frequency pieces together, since there is no extra decay. Herr, Tataru and Tzvetkov \cite{HTT11} overcame this difficulty by applying a trilinear Strichartz estimate, which will give extra decay when the frequencies separat away from each other. Their argument also relies on $U^p,V^p$ based critical function spaces, which is a effective substitution of Bourgain spaces for critical problems, see \cite{KoTa05,KoTa07,HHK09,HTT10}.
\end{remark}

The main purpose of this paper is to generalize Herr, Tataru and Tzvetkov's result, which only consider \eqref{eq:nls} with $n=3$ and $k=2$, to more general setting \eqref{eq:nls}. And the main results are,

\begin{theorem}[Main Results]\label{main} For the initial value problem \eqref{eq:nls}, we have
\begin{itemize}
\item If $n= 1$ and $k\ge 3$, \eqref{eq:nls} is local well-posed for $\phi \in H^s(\T)$, provided $s\ge s_{1,k}$.
\item If $n= 2, 3, 4$ and $k\ge 2$, \eqref{eq:nls} is local well-posed for $\phi \in H^s(\T^n)$, provided $s\ge s_{n,k}$.
\item If $n\ge5$ and $k\ge 1$, \eqref{eq:nls} is local well-posed for $\phi \in H^s(\T^n)$, provided $s\ge s_{n,k}$.
\end{itemize}
Additionally, if $s_{n,k} \le 1$, then the \eqref{eq:nls} is globally well-posed for sufficiently small initial data $\phi \in H^1(\T^n)$.
\end{theorem}

\begin{remark}
For initial value problem \eqref{eq:nls}, this theorem extend Bourgain's subcritical results to critical spaces, and also generalize Herr, Tataru and Tzvetkov's result which only consider 3-dimension case with quintic nonlinearity. The condition $s_{n,k}\le 1$ implies that $n=1,2$ for all $k$, or $n=3$ for $k=2$. 
\end{remark}

\begin{remark}
Our main new tool is a multi-linear Stricartz estimate for solutions to linear Schr\"odinger equation, which ensure a decay when the high-low frequencies separate away from each other. We also need the critical function spaces based on $U^p$ and $V^p$ to make sure there is no loss in the inhomogeneous estimate. The multi-linear Stricartz estimates is the generalization of the trilinear ones used by Herr, Tataru and Tzvetkov \cite{HTT11}, on which our approach heavily relay.
\end{remark}

\begin{remark}
Our results couldn't cover the case for $n=4$ and $k=1$, which is the energy critical cubic Schr\"odinger equation
\begin{equation}
	\label{eq:nls2}
	\begin{cases}
		i u_t - \Delta u  = \pm |u|^{2} u  \\
		u|_{t= 0} = \phi(x),
	\end{cases}~(x, t)  \in \T^4 \times \R.
\end{equation}
Since in dimension $n=4$ the $L^4$ Strichartz estimate is the endpoint case, see \cite[Proposition 3.6]{B93a}. But our argument relay on a perturbation of $L^p$ with the index $p$ around 4, see for example \cite[Proposition 3.5]{HTT11}, where they abandon $L^6$ but adopt $L^p$ and $L^q$ with $p<6<q$. Such argument is crucial to gain decay between high-low frequencies. 
\end{remark}

The outline of the paper is as follows: In Section \ref{pre} we give some notations and preliminary lemmas. In Section \ref{str} we prove multi-linear Strichartz type estimates. In Section \ref{main_non} we show the main nonlinear estimates, which will imply our main results.


\section{Preliminaries}\label{pre}
\noindent
In this section we will give some notations and preliminary lemmas that will be used in this paper.

We write $A\lesssim B$ to indicate that there is a constant $C>0$ such that $A\le C B$, and we denote $A\thickapprox B$ when $A\lesssim B\lesssim A$. Define the spatial Fourier coefficients
\[
\widehat{f}(\xi):=c_1\int_{ \T^{n}} e^{-ix\cdot \xi}
f(x) \;dx,\;\quad \xi \in \Z^{n},
\]
and the space time Fourier transform
\[
\mathcal{F}u(\tau,\xi):= c_2 \int_{ \R\times \T^{n}}
e^{-i(x\cdot \xi+t\tau)} u(t,x) \;dt dx,\; (\tau,\xi) \in
\R\times\Z^{n},
\]
where $c_1,c_2$ are explicit constant.

Let $\psi\in C_0^\infty((-2,2))$ be a non-negative, even function satisfies
$\psi(s)=1$ for $|s|\leq 1$. Let $N\in 2^{\N}$ be a dyadic number and define
\[
\psi_1(\xi)=\psi(|\xi|); \quad  \psi_N(\xi)=\psi_1(N^{-1} \xi)- \psi_1(2N^{-1}\xi), \quad \text{for } N\geq 4.
\]
For $f\in L^2(\T^n)$, we define Littlewood-Paley projection operator
\[
\widehat{P_N f}(\xi) = \psi_N (\xi) \widehat f(\xi),
\]
and define $P_{\leq N}:=\sum_{1\leq M\leq N}P_M$. More generally, for a set $S\subset \Z^{n}$ and $\chi_S$ denoting the characteristic function of $S$, we define the Fourier projection operator
\[
\widehat{P_S f}(\xi) = \chi_S (\xi) \widehat f(\xi).
\]
Let $s \in \R$, we define the Sobolev space $H^s(\T^n)$ by the norm
\[
\|f\|_{H^s(\T^n)}:= \left(\sum_{N \geq 1}N^{2s} \|P_N
  f\|_{L^2(\T^n)}^2\right)^{\frac12}.
\]

\subsection*{The function spaces} The $U^p$ and $V^p$ spaces will be used to construct our main function spaces, but we only recall some of their properties and refer the reader to \cite[Section 2]{HHK09} for detailed definitions and proofs.

For $s \in \R$ we let $U^p_\Delta H^s $ and $V^p_\Delta H^s$ be
the spaces with norms
\begin{equation}\label{eq:delta_norm}
\| u\|_{U^p_\Delta H^s} = \| e^{-it \Delta} u\|_{U^p(\R,H^s)},
    \qquad
\| u\|_{V^p_\Delta H^s} = \| e^{-it \Delta} u\|_{V^p(\R,H^s)}.
\end{equation}
And the following continuous embedding (see \cite[Proposition 2.2]{HHK09}) will be used in this paper
\begin{equation}\label{embed}
 U^p_\Delta H^s \hookrightarrow U^q_\Delta H^s \hookrightarrow L^{\infty}(\R;H^s), \quad \text{ for all } 1\le p<q<\infty.
\end{equation}
Spaces of such type have been successfully used as substitutions for
$X^{s,b}$ spaces which are still effective at critical scaling, see
for instance \cite{KoTa05,KoTa07,HHK09,HTT10,HTT11}. Now we are ready to define the main resolution spaces. Let $X^{s}$ be the space defined by the norm
\begin{equation}\label{X}
 \|u\|_{X^{s}}:=\left(\sum_{\xi \in \Z^n}\lb \xi\rb ^{2s}
 \|e^{it|\xi|^2}\widehat{u(t)}(\xi)\|_{U^2_t}^2\right)^{\frac12}
\end{equation}
and $Y^{s}$ be the space defined by the norm
\begin{equation}\label{Y}
 \|u\|_{Y^{s}} :=\left(\sum_{\xi\in \Z^n}
 \lb\xi\rb ^{2s} \|e^{it|\xi|^2}\widehat{u(t)}(\xi)\|_{V^2_t}^2\right)^{\frac12}.
\end{equation}
Also we can define the spaces $X(I)$ and $Y(I)$ to be the restriction of the original spaces on the time interval $I\subset \R$. The following continuous embeddings hold (see c.f. \cite{HTT11}):
  \begin{align}\label{embedding}
  U^2_\Delta H^s \hookrightarrow X^s \hookrightarrow Y^s
  \hookrightarrow V^2_\Delta H^s.
  \end{align}
Let $\Z^n = \cup C_k$ be a partition of $\Z^n$, then
  \begin{equation}
    \left(\sum_k \| P_{C_k} u\|_{V^2_\Delta H^s}^2\right)^{\frac12} \lesssim  \| u\|_{Y^s}.
  \end{equation}
We will use a interpolation property of the spaces $U^p$ and $V^p$, reader could find the linear and trilinear version in \cite[Proposition 2.20]{HHK09} and \cite[Lemma 2.4]{HTT11}, the proof for the following multi-linear version is similar.
\begin{lemma}\label{interpolation}
  Let $q_1,\cdots,q_k>2$ and
  \[
    T:U^{q_1}\times \cdots\times U^{q_{k}}\to E
  \] 
  be a bounded, multi-linear operator satisfy
  \begin{align*}
    \|T(u_1,\cdots,u_k)\|_{L^2} \leq &C_1 \prod_{j=1}^k\|u_j\|_{U^{q_j}}\\
    \|T(u_1,\cdots,u_3)\|_{L^2} \leq &C_2 \prod_{j=1}^k\|u_j\|_{U^2},
  \end{align*}
with $0<C_2<C_1$. Then we have
  \[
    \|T(u_1,\cdots,u_3)\|_{L^2} \ls C_2 \Big(\ln\frac{C_1}{C_2}+1\Big)^k\prod_{j=1}^k\|u_j\|_{V^2}.
  \]
\end{lemma}

The following two linear estimates could be found in \cite[Proposition 2.9 and 2.10]{HTT11}.
\begin{lemma}\label{linear}
  Let $s\geq 0$, $0<T\leq \infty$ and $\phi \in H^s(\T^n)$. Then
  \begin{equation}\label{eq:linear}
    \|e^{it\Delta}\phi\|_{X^s([0,T))}\leq \|\phi\|_{H^s}.
  \end{equation}
\end{lemma}

Let $f\in L^1_{loc}([0,\infty);L^2(\T^n))$ and define
\begin{equation}\label{eq:duhamel}
  \mathcal{I}(f)(t):=\int_{0}^t e^{i(t-s)\Delta} f(s) ds,
\end{equation}
for $t \geq 0$ and $\mathcal{I}(f)(t)=0$ otherwise. We have the following linear estimate for the Duhamel term.
\begin{lemma}\label{inhom}
  Let $s \geq 0$ and $T>0$. For $f \in L^1([0,T);H^s(\T^n))$ we have $\mathcal{I}(f)\in X^s([0,T))$ and
  \begin{equation}\label{eq:inhom_est}
    \|\mathcal{I}(f)\|_{X^s([0,T))} \leq \sup_{v \in Y^{-s}([0,T)): \|v\|_{Y^{-s}}=1}
    \left|\int_0^T \int_{\T^3}f(t,x)\overline{v(t,x)} dx dt\right|.
  \end{equation}
\end{lemma}

\section{Strichartz estimates}\label{str}
In this section, we will use Bourgain's $L^p$ Stricartz estimates to prove a multilinear Strichartz estimate, which is the main new tool in this paper. The key observation is that the period Stichartz estimates only related to the measure of the frequency supports, see \cite{B93a,HTT11} for more detailed discussion.

\begin{definition}
We say that $(n,p)\in \N\times \R$ is admissible pairs if
  \begin{align}\label{admis}
    n=&1, \quad \, p>6;\nonumber\\
    n=&2, 3, \quad \, p>4;\\
    n\ge& 4,\quad\, p\ge\frac{2(n+4)}{n}\nonumber.
  \end{align}
We will only use Strichartz estimates with $(n,p)$ in this range, since in this range the following Strichartz estimates are sharp.
\end{definition}
We first recall Bourgain's fundamental Strichartz estimates. Let $N$ be a dyadic number, and $\mathcal{C}_N$ denote the collection of cubes $C\subset \Z^n$ of side-length $N\geq 1$ with arbitrary center and orientation.
\begin{lemma}[Bourgain \cite{B93a}]\label{prop:str}For all $N\geq 1$ and admissible pairs $(n,p)$, we have
  \begin{align}\label{bourgain_Str}
    \|P_N e^{it \Delta}\phi\|_{L^p(\T\times \T^n)}
    \ls N^{\frac{n}2-\frac{n+2}p} \|P_N \phi\|_{L^2(\T^n)}.
  \end{align}
More generally, for all $C \in \mathcal{C}_N$ and admissible pairs $(n,p)$ we have
  \begin{equation}\label{bourgain_Str_cubes}
    \|P_C e^{it \Delta}\phi\|_{L^p(\T\times \T^n)}
    \ls N^{\frac{n}2-\frac{n+2}p} \|P_C \phi\|_{L^2(\T^n)}.
  \end{equation}
\end{lemma}
\begin{remark}
Under the admissible condition \eqref{admis}, the Strichartz estimates \eqref{bourgain_Str} and \eqref{bourgain_Str_cubes} are sharp, which means that one could apply those estimates to get `almost' critical well-posedness (see c.f. Theorem \ref{Bourgain}) ignoring log divergence. The estimate \eqref{bourgain_Str_cubes} shows that it only depend on the measure of the frequency's support, not the position.
\end{remark}

In order to employ more orthogonality between different frequency pieces of Schr\"odnger evolution, we need further decompose the frequency piece. Let $\mathcal{R}_{M}(N)$ be the collection of all sets in $\Z^n$ which are given as the intersection of a cube of side length $2N$ with strips of width $2M$, i.e. the collection of all sets of the form
\[
  (\xi_0+[-N,N]^n)\cap \{ \xi \in \Z^n: |a\cdot \xi -A|\leq M \}
\]
with some $\xi_0 \in \Z^n, a \in \R^n$, $|a|=1$, $A \in \R$.  For all $1\leq M\leq N$ and $R \in \mathcal{R}_M (N)$, by Bernstein inequality we have
\begin{equation}\label{bernstein}
  \|P_R  e^{it \Delta}\phi\|_{L^\infty (\T\times \T^n)}
  \ls M^{\frac12}N^{\frac{n-1}{2}} \|P_R \phi\|_{L^2(\T^n)}.
\end{equation}
By H\"older's inequality with the estimates \eqref{bourgain_Str_cubes} and \eqref{bernstein} we get
\begin{lemma}\label{local_str}
For all $1\leq M\leq N$ and $R \in \mathcal{R}_M(N)$ we have
  \begin{equation}\label{bourgain_Str_impr}
    \|P_R e^{it \Delta}\phi\|_{L^p(\T\times \T^n)}
    \ls N^{\frac{n}2-\frac{n+2}p} \left(\frac{M}{N}\right)^{\delta}\|P_R \phi\|_{L^2(\T^n)},
  \end{equation}
where $(n,p)$ are admissible pairs defined in \eqref{admis}, and for $n=1$, $0<\delta<\frac12-\frac3p$; for $n=2,3$, $0<\delta<\frac12-\frac2p$; for $n\ge 4$, $0<\delta<\frac12- \frac{n+3}{np}$.
\end{lemma}

We conclude this section with our key multi-linear Strichartz estimates:

\begin{theorem}[Multi-linear Strichartz estimate]\label{main_th}
Assume $N_1\ge N_2\ge \cdots \ge N_{k+1} \ge 1$, $u_j =e^{it\Delta}\phi_j$ and
  \begin{align}\label{admis1}
    n=1,\, k\ge3; \quad
    n=2, 3,4, \, k\ge2; \quad
    n\ge 5,\, k\ge 1.
  \end{align}
Then there exist $\delta'>0$, such that
  \begin{equation}\label{eq:multi}
    \begin{split}
     \Big\|\prod_{j=1}^{k+1}P_{N_j} u_j\Big\|_{L^2(\T\times\T^n)} \ls& \left(\frac{N_{k+1}}{N_1}+\frac1{N_2}\right)^{\delta'}\|P_{N_1}\phi_1\|_{L^2}\\
     &\prod_{j=2}^{k+1}N^{s_{n,k}}_j \|P_{N_j}\phi_j\|_{L^2}.
    \end{split}
  \end{equation}
\end{theorem}

\begin{remark}
Theorem \ref{main_th} is the main new tool in this paper, the proof base on the idea from \cite{HTT11}, where the trilinear Strichartz estimate is studied. The condition \eqref{admis1} is also the condition for the Theorem \ref{main}, which comes from the admissible condition \eqref{admis}.
\end{remark}

\begin{proof}
By orthogonality, it suffices to prove
  \begin{equation*}
    \begin{split}
      \Big\|P_CP_{N_1}u_1\prod_{j=2}^{k+1}P_{N_j} u_j\Big\|_{L^2(\T\times\T^n)} \ls{}& \left(\frac{N_{k+1}}{N_1}+\frac1{N_2}\right)^{\delta'}\|P_C P_{N_1}\phi_1\|_{L^2} \\&\prod_{j=2}^{k+1}N^{s_{n,k}}_j \|P_{N_j}\phi_j\|_{L^2},
    \end{split}
  \end{equation*}
where for all $C \in \mathcal{C}_{N_2}$. Now fix one $C$ and let $\xi_0$ be the center of $C$.  We partition $C = \cup R_l$ into disjoint strips with width $M=\max\{N_2^2/N_1,1\}$, which are all orthogonal to $\xi_0$,
  \[
   R_l=\Big\{\xi\in C ;\,\xi\cdot \xi_0 \in [|\xi_0| M l, |\xi_0| M(l+1)) \Big\}, \qquad |l| \approx N_1/M.
  \]
It is easy to see that $R_l \in \mathcal{R}_M(N_2)$, and we decompose
  \[
   P_C P_{N_1} u_1  \prod_{j=2}^{k+1}P_{N_j} u_j = \sum_l P_{R_l} P_{N_1} u_1  \prod_{j=2}^{k+1}P_{N_j} u_j
  \]
and we will show that the sum are almost orthogonal in $L^2(\T \times \T^n)$. Indeed, for $\xi_1 \in R_l$ we have
  \[
    |\xi_1|^2 = \frac1{|\xi_0|^{2}}|\xi_1 \cdot \xi_0|^2 + |\xi_1-\xi_0|^2 - \frac1{|\xi_0|^{2}}|(\xi_1- \xi_0) \cdot \xi_0|^2 = M^2 l^2 + O(M^2 l),
  \]
since $N_2^2 \lesssim M^2 l$. The factor $\prod_{j=2}^{k+1}P_{N_j} u_j$ only alter the time frequency by at most $O(N_2^2)$. Hence the expressions $P_{R_l} P_{N_1} u_1 \prod_{j=2}^{k+1}P_{N_j} u_j$ are localized at time frequency $M^2 l^2 + O(M^2 l)$ and thus are almost orthogonal,
  \[
    \Big\|  P_C P_{N_1} u_1  \prod_{j=2}^{k+1}P_{N_j} u_j \Big\|_{L^2}^2 \approx \sum_l \|P_{R_l} P_{N_1} u_1  \prod_{j=2}^{k+1}P_{N_j} u_j\|_{L^2}^2
  \]

If $k=1$, which implies that $n\ge 5$. Let $p\in [\frac{2(n+4)}{n},4)$ and $q\in (4,\frac{n+4}{2}]$ satisfy $\frac{1}{2}=\frac1p + \frac1q$, then by H\"older inequality we have
\[
  \|P_{R_l} P_{N_1} u_1  P_{N_2} u_2\|_{L^2} \le \|P_{R_l} P_{N_1} u_1\|_{L^p} \|P_{N_2} u_2\|_{L^q},
\]
then we apply Lemma \ref{prop:str} and Lemma \ref{local_str} to continue with
\begin{align*}
\ls{}& N_2^{\frac n2 - \frac{n+2}{p}}N_{2}^{\frac n2- \frac{n+2}q  -s_{n,1}} \Big(\frac{M}{N_2} \Big)^{\delta}\|P_{R_l}P_{N_1}\phi_1\|_{ L^2} N^{s_{n,1}}_2\|P_{N_2}\phi_2\|_{ L^2}.
\end{align*}
It is easy to see that $N_2^{\frac n2 - \frac{n+2}{p}}N_{2}^{\frac n2- \frac{n+2}q -s_{n,1}} =1$ and $\frac{M}{N_2} = \left(\frac{N_2}{N_1}+\frac1{N_2} \right)$, thus we finish the proof for this case by summing up the squares with respect to $l$.

If $k\ge 2$, let $p_{n,k}=(n+2)k$, then we have
\[
  \|P_N e^{it \Delta}\phi\|_{L^{p_{n,k}}(\T\times \T^n)}
  \ls N^{s_{n,k}} \|P_N \phi\|_{L^2(\T^n)}.
\]
Let $p$ satisfies the conditions
\begin{equation}\label{eq:p_var}
  \begin{split}
   6<p<\frac{12k}{k+2} \text{ for } n=1;\\
   4<p<\frac{4k(n+2)}{nk+2} \text{ for } n=2,3,4; \\
   \frac{2(n+4)}{n}<p<\frac{4k(n+2)}{nk+2}\text{ for } n\ge5;
  \end{split}
\end{equation}
The existence of such $p$ is implied by \eqref{admis1}; and the lower bound of $p$ implies that each $(n,p)$ is admissible, and the higher bound  guarantees that
\begin{align}\label{p_range}
  n-\frac{2(n+2)}{p}-s_{n,k} <0,
\end{align}
which will be used latter. By H\"older's inequality with some $q$ such that
\begin{align}\label{eq:scaling}
  \frac2p+\frac{k-2}{p_{n,k}}+\frac1q=\frac12,
\end{align}
where $(n,q)$ is also admissible, then we have
\begin{align*}
  &\Big\|P_{R_l} P_{N_1} u_1 \prod_{j=2}^kP_{N_j} u_j\Big\|_{L^2}\\
 \ls{}&\|P_{R_l} P_{N_1} u_1 \|_{L^p} \|P_{N_2}u_2\|_{L^p} \prod_{j=3}^k
  \|P_{N_j}u_j\|_{L^{p_{n,k}}} \|P_{N_{k+1}}u_{k+1}\|_{L^q},
\end{align*}
applying Lemma \ref{prop:str} and Lemma \ref{local_str} to continue with
\begin{align*}
  \ls{}&N_2^{n-\frac{2(n+2)}{p}-s_{n,k}}N_{k+1}^{\frac n2- \frac{n+2}q -s_{n,k}}\Big(\frac{M}{N_2} \Big)^{\delta}\|P_{R_l}P_{N_1}\phi_1\|_{ L^2} \\
  &\times N^{s_{n,k}}_2\|P_{N_2}\phi_2\|_{ L^2} N^{s_{n,k}}_{k+1}\|P_{N_{k+1}}\phi_{k+1} \|_{L^2} \prod_{j=3}^k N^{s_{n,k}}_j \|P_{N_j}\phi_j\|_{L^2},
\end{align*}
where $s_{n,k} =\frac n2-\frac1k$ defined by \eqref{ind}. In view of \eqref{eq:scaling}, we have
\[
  -n+\frac{2(n+2)}{p}+s_{n,k} = \frac n2-\frac{n+2}q-s_{n,k},
\]
Then since $-n+\frac{2(n+2)}{p}+s_{n,k}> 0$ by \eqref{p_range}, we finially get
\begin{align*}
  & \Big\|  P_{R_l} P_{N_1} u_1  \prod_{j=2}^{k+1}P_{N_j} u_j \Big\|_{L^2} \\
  \ls &\Big(\frac{N_{k+1}}{N_2}\Big)^{- n+\frac{2(n+2)}{p}+s_{n,k}} \left(\frac{M}{N_2}\right)^{\delta}\|P_{R_l} P_{N_1}\phi_1\|_{L^2} \\
  &\times \prod_{j=2}^{k+1}N^{s_{n,k}}_j \|P_{N_j}\phi_j\|_{L^2},
\end{align*}
where $0<\delta \ll1$ and $\frac{M}{N_2} = \left(\frac{N_2}{N_1}+\frac1{N_2} \right)$. We can select $p$ such that
\[
  - n+\frac{2(n+2)}{p}+s_{n,k} = \delta,
\]
and continue the estimate with
\[
  \ls\left(\frac{N_{k+1}}{N_2} \left(\frac{N_2}{N_1}+\frac1{N_2}\right)\right)^{\delta}\|P_{R_l} P_{N_1}\phi_1\|_{L^2}  \prod_{j=2}^{k+1}N^{s_{n,k}}_j \|P_{N_j}\phi_j\|_{L^2}.
\]
Then \eqref{eq:multi} follows by summing up the squares with respect to $l$.
\end{proof}

\section{Nonlinear estimates and the main results}\label{main_non}

\noindent
Before we proof the main nonlinear estimates, we need a embedding result, where the Strichartz norm could be bounded by $U^p$ type spaces. Such extension result is well known for Bourgain type spaces.
\begin{coro}
 For all $N\geq 1$, $C \in \mathcal{C}_N$ and admissible pairs $(n,p)$ we have
  \begin{equation}\label{eq:str_emb_cubes}
    \|P_C u\|_{L^p(\T\times \T^n)}
    \ls  N^{\frac{n}2-\frac{n+2}p} \|P_C u\|_{U^p_{\Delta} L^2}.
  \end{equation}
\end{coro}
\begin{proof}
We refer the reader to \cite[Proposition 2.19]{HHK09} and \cite[Corollary 3.2]{HTT11} for the detailed proof.
\end{proof}

Now we are ready to proof our main $L^2$ multi linear estimates.
\begin{prop}\label{prop:L2}
Assume $n=1$, $k\ge 3$; $n=2,3,4$, $k\ge2$ and  $n\ge5$, $k\ge1$. For any $N_i\geq N_{i+1}\geq 1$ for all $i\in \{1,2,\ldots,k-1\}$ and any interval $I\subset [0,2\pi]$. Then there exists a $\delta>0$ such that
  \begin{equation}\label{eq:l2_est_linear}
    \Big\|\prod_{j=1}^{k+1}P_{N_j} u_j\Big\|_{L^2(I\times \T^n)}\ls \Big(\frac{N_{k+1}}{N_1} +\frac{1}{N_2}\Big)^\delta \|P_{N_1}u_1\|_{Y^0}\prod_{j=2}^{k+1}\|P_{N_j}u_j\|_{Y^{s_{n,k}}},
  \end{equation}
where $s_{n,k}= \frac n2-\frac1k$ is the critical index.
\end{prop}
\begin{proof}
It is enough to prove
  \begin{align}\label{main_1}
    \Big\|P_{N_1} P_C u_1\prod_{j=2}^{k+1}P_{N_j} u_j\Big\|_{L^2}\ls &\|P_{N_1}u_1\|_{Y^0} \Big(\frac{N_{k+1}}{N_1}+ \frac{1}{N_2}\Big)^\delta \nonumber\\
    &\prod_{j=2}^kN^{s_{n,k}}_j\|P_{N_j}u_j\|_{Y^0},
  \end{align}
where $I=[0,2\pi]$ and $C \in \mathcal{C}_{N_2}$. In view of \eqref{embedding}, we could replace $Y^0$ by $V^2_\Delta L^2$. Then by Lemma~\ref{interpolation}, \eqref{main_1} follows from the following two multi-linear estimates:
  \begin{equation}\label{main_12}
    \begin{split}
     \Big\|P_C P_{N_1} u_1 \prod_{j=2}^{k+1}P_{N_j} u_j\Big\|_{L^2} \ls{} \left(\frac{N_{k+1}}{N_2}\right)^{\delta'}\prod_{j=2}^{k+1}N^{s_{n,k}}_j
      \|P_{N_j}u_j\|_{U^{r_n}_\Delta L^2}
    \end{split}
  \end{equation}
where $r_1=6, r_2 = r_3 =4, r_n =\frac{2(n+4)}{n}$ for $n\ge 4$, and
  \begin{equation}\label{main_13}
    \Big\|P_C P_{N_1} u_1  \prod_{j=2}^{k+1}P_{N_j} u_j\Big\|_{L^2}
    \ls   \left(\frac{N_{k+1}}{N_1}+\frac{1}{N_2}\right)^{\delta'} \prod_{j=2}^{k+1}N^{s_{n,k}}_j \|P_{N_j}u_j\|_{U^2_\Delta L^2},
  \end{equation}
for some $\delta'>0$.

We first consider \eqref{main_12}. If $k=1$, which implies that $n\ge 5$. Let $p\in [\frac{2(n+4)}{n},4)$ and $q\in (4,\frac{n+4}{2}]$ satisfy $\frac{1}{2}=\frac1p + \frac1q$, then by H\"older inequality we have
\[
  \|P_{C} P_{N_1} u_1  P_{N_2} u_2\|_{L^2} \le \|P_{C} P_{N_1} u_1\|_{L^p} \|P_{N_2} u_2\|_{L^q},
\]
then we apply Lemma \ref{prop:str} to continue with
\begin{align*}
  \ls{}&N_2^{\frac n2 - \frac{n+2}{p}}N_{2}^{\frac n2- \frac{n+2}q -s_{n,1}}\|P_{C}P_{N_1}\phi_1\|_{ L^2} N^{s_{n,1}}_2\|P_{N_2}\phi_2\|_{ L^2}.
\end{align*}
Since $N_2^{\frac n2 - \frac{n+2}{p}}N_{2}^{\frac n2- \frac{n+2}q -s_{n,1}} =1$, thus we get \eqref{main_12} with $k=1$.

If $k\ge 2$, by H\"older's inequality with $p_{n,k}$, $p$ and $q$ being the same ones in the proof of Theorem \ref{main_th} and we get
\begin{align*}
  &\Big\|P_C P_{N_1} u_1 \prod_{j=2}^kP_{N_j} u_j\Big\|_{L^2}\\
  \leq&\|P_C P_{N_1} u_1 \|_{L^p} \|P_{N_2}u_2\|_{L^p} \prod_{j=3}^k
  \|P_{N_j}u_j\|_{L^{p_{n,k}}} \|P_{N_{k+1}}u_{k+1}\|_{L^q},
\end{align*}
then by \eqref{eq:str_emb_cubes}, we continue with
\begin{align*}
  \ls{}&N_2^{n-\frac{2(n+2)}{p}-s_{n,k}}N_{k+1}^{\frac n2-\frac{n+2}q-s_{n,k}} \|P_CP_{N_1}u_1\|_{U^{p}_\Delta L^2} N^{s_{n,k}}_2\|P_{N_2}u_2\|_{U^{p}_\Delta L^2}\\
  &\times N^{s_{n,k}}_{k+1}\|P_{N_{k+1}}u_{k+1}\|_{U^{q}_\Delta L^2} \prod_{j=3}^kN^{s_{n,k}}_j
  \|P_{N_j}u_j\|_{U_{\Delta}^{p_{n,k}}L^2},
\end{align*}
By \eqref{embed} we have $U^{r_n}_\Delta L^2 \hookrightarrow U^{p}_\Delta L^2\hookrightarrow U^{p_{n,k}}_\Delta L^2\hookrightarrow U^{q}_\Delta L^2$. In view of the choice of $p,q$ we have
\begin{align*}
  n-\frac{2(n+2)}{p}-s_{n,k} <&0,\\
  -n+\frac{2(n+2)}{p}+s_{n,k} =& \frac n2-\frac{n+2}q-s_{n,k}
\end{align*}
Then we finish the proof by letting $\delta'=-n+\frac{2(n+2)}{p}+s_{n,k}$.

In view of the atomic structure of the $U^2$ spaces (see e.g. \cite[Proposition 2.19]{HHK09}), the second bound \eqref{main_13} reduce to: for $u_j=e^{it\Delta} \phi_j$ we have
\begin{equation*}
     \Big\|P_C P_{N_1} u_1  \prod_{j=2}^{k+1}P_{N_j} u_j \Big\|_{L^2}
    \ls  \left(\frac{N_{k+1}}{N_1}+\frac{1}{N_2}\right)^{\delta'} \prod_{j=2}^{k+1}N^{s_{n,k}}_j \|P_{N_j}u_j\|_{L^2},
\end{equation*}
which is just the Multi-linear Strichartz estimate \eqref{eq:multi}.
\end{proof}

In order to prove Theorems \ref{main} we need the following nonlinear estimate, which is implied by Proposition \ref{prop:L2}. The argument is standard, see c.f. \cite{HHK09} and \cite{HTT11}.
\begin{prop}\label{prop:nonlinear}
Assume $(n,k)$ as in Proposition \ref{prop:L2} and $s\geq s_{n,k}$, $0<T\leq 2\pi$, and $u_l\in   X^s([0,T))$, $l=1,\ldots,2k+1$. Then
\begin{equation}\label{eq:non}
  \Big\|\mathcal{I}(\prod_{k=1}^{2k+1} \widetilde{u}_l)\Big\|_{X^s([0,T))}
  \ls  \sum_{j=1}^{2k+1}\|u_j\|_{X^s([0,T))}\prod_{\genfrac{}{}{0pt}{} {l=1}{l\not=j}}^{2k+1}\|u_l\|_{X^{s_{n,k}}([0,T))},
\end{equation}
where $\widetilde{u}_l$ denotes either $u_l$ or $\overline{u}_l$.
\end{prop}

\begin{proof} We can assume that $\widetilde{u}_l \in H^{\infty}(\T^n)$ by density argument. Denote $I=[0,T)$, and let $N\geq 1$. Proposition \ref{inhom} implies
  \[
    \Big\|\mathcal{I}(\prod_{l=1}^{2k+1} \widetilde{u}_l)\Big\|_{X^s(I)}\leq \sup_{\genfrac{}{}{0pt}{}{v \in Y^{-s}(I):}{\|v\|_{Y^{-s}}=1}} \int_0^{2\pi} \int_{\T^n}\prod_{l=1}^{2k+1} \widetilde{u}_l \, \overline{v} dx dt.
  \]
Denote $u_0 =  v$.  Then it is sufficient to prove
  \begin{equation*}
    \left| \int_{I \times \T^n} \prod_{l=0}^{2k+1} \tilde u_l \ dx dt\right|
    \lesssim  \| u_0\|_{Y^{-s}(I)}  \sum_{j=1}^{2k+1}\left(\|u_j\|_{X^s(I)} \prod_{\genfrac{}{}{0pt}{}{l=1}{l\not=j}}^{2k+1}\|u_l\|_{X^{s_{n,k}}(I)}\right)
  \end{equation*}
In view of the definition of $X^s(I)$, for the above we only need to prove
 \begin{equation}\label{2k+1lin_g}
    \left| \int_{I \times \T^n} \prod_{l=0}^{2k+1} \tilde u_l \ dx dt\right|
    \lesssim  \| u_0\|_{Y^{-s}}  \sum_{j=1}^{2k+1}\left(\|u_j\|_{X^s} \prod_{\genfrac{}{}{0pt}{}{l=1}{l\not=j}}^{2k+1}\|u_l\|_{X^{s_{n,k}}}\right).
 \end{equation}
  We decompose $u_l$ into frequency dyadic pieces,
  \[
   \widetilde{u}_l=\sum_{N_l\geq 1} P_{N_l} \widetilde{u}_l.
  \]
If $N_1\ge N_2\ge \cdots\ge N_{2k+1}$, and
  \[
    \left| \int_{I \times \T^n} \prod_{l=0}^{2k+1} P_{N_l} \tilde u_l \ dx dt\right| \neq 0,
  \]
then we must have that
  \[
    N_1 \thickapprox \max\{N_0,N_2\}.
  \]
Then, by the Cauchy-Schwarz inequality and symmetry it suffices to show that
  \begin{equation}\label{2k+1lin_ga}
  \begin{split}
     S&= \sum_{ \mathcal N}\Big\| \prod_{l=0}^k P_{N_{2l+1}}\widetilde{u}_{2l+1}\Big\|_{L^2} \Big\|\prod_{l=0}^k P_{N_{2l}}\widetilde{u}_{2l}\Big\|_{L^2}
     \\& \lesssim \| u_0\|_{Y^{-s}} \sum_{j=1}^{2k+1}\|u_j\|_{X^s} \prod_{\genfrac{}{}{0pt}{}{l=1}{l\not=j}}^{2k+1}\|u_l\|_{X^{s_{n,k}}},
  \end{split}
  \end{equation}
where $\mathcal N$ is a subset of
  \[
    \{(N_0,N_1,\ldots,N_{2k+1})\in\Z^{2k+2}; N_i \text{ are dyadic numbers}\}
  \]
satisfying
  \[
   N_1\ge N_2\ge \cdots\ge N_{2k+1}, \quad \max\{N_0,N_2\} \approx N_1.
  \]
We subdivide the sum into two parts $S = S_1+S_2$:

First assume that $N_2\leq N_0 \approx N_1$, then Proposition \ref{prop:L2} implies
  \[
  S_1 \lesssim \sum_{S_1} N_1^{-2s_{n,k}} \left(\frac{N_{2k+1}}{N_1}+ \frac{1}{N_3}\right )^\delta \left(\frac{N_{2k}}{N_0}+ \frac{1}{N_2}\right )^\delta \prod_{l=0}^{2k+1} N_l^{s_{n,k}}
  \|P_{N_l}u_l\|_{Y^0}.
  \]
Sum $N_2$, $\cdots$, $N_{2k+1}$ together by Cauchy-Schwarz, and obtain
  \[
  S_1 \lesssim \sum_{N_0 \approx N_1} \| P_{N_0}u_0\|_{Y^0} \|
  P_{N_1}u_1\|_{Y^0} \prod_{l=2}^{2k+1} \|u_l\|_{Y^{s_{n,k}}}.
  \]
Then sum $N_1$ by Cauchy-Schwarz we get
  \[
  S_1 \lesssim \| u_0\|_{Y^{-s}} \| u_1\|_{Y^{s}} \prod_{l=2}^{2k+1} \|u_l\|_{Y^{s_{n,k}}},
  \]
  as needed.

Now if $N_0 \leq N_2 \approx N_1$, then by Proposition \ref{prop:L2} we have
  \[
    S_2 \lesssim \sum_{S_2} N_1^{-2s_{n,k}} \left(\frac{N_{2k+1}}{N_1}+
    \frac{1}{N_3}\right )^\delta  \prod_{l=0}^{2k+1} N_l^{s_{n,k}}
  \|P_{N_l}u_l\|_{Y^0}
  \]
  By Cauchy-Schwarz we sum $N_3$, $\cdots$  $N_{2k+1}$ together and obtain
  \[
  S_2 \lesssim \sum_{N_0 \leq N_1 \approx N_2} N_0^{s_{n,k}} \|
  P_{N_0}u_0\|_{Y^0} \| P_{N_1}u_1\|_{Y^0} \|
  P_{N_2}u_2\|_{Y^0} \prod_{l=3}^{2k+1} \|u_l\|_{Y^{s_{n,k}}}.
  \]
Then apply Cauchy-Schwarz with respect to $N_0$ to obtain
  \[
  S_2 \lesssim \sum_{N_1\approx N_2} N_1^{s+s_{n,k}} \|u_0\|_{Y^{-s}} \|
  P_{N_1}u_1\|_{Y^0} \| P_{N_2}u_2\|_{Y^0} \prod_{l=3}^{2k+1} \|u_l\|_{Y^{s_{n,k}}}.
  \]
  Finally, we apply Cauchy-Schwarz with respect to $N_1$ to obtain
  \[
  S_2 \lesssim \| u_0\|_{Y^{-s}} \| u_1\|_{Y^{s}} \prod_{l=2}^{2k+1} \|u_l\|_{Y^{s_{n,k}}}.
  \]
  Thus the proof is complete.
\end{proof}

\begin{proof}[Proof of Theorem \ref{main}] The general argument of the proof is
  well-known, see e.g. \cite{HHK09,HTT11}. We only study the critical case $s=s_{n,k}$ for small data here for an example, and refer the readers to \cite{HTT11} for more details.

Our aim is to solve the equation
\begin{align}\label{int_eq}
u= &e^{it\Delta} \phi + c\int_{0}^t e^{i(t-s)\Delta} (|u|^{2k}u)(s) ds \nonumber\\
=&e^{it\Delta} \phi + c\mathcal{I}(|u|^{2k}u).
\end{align}
From \eqref{eq:linear} we have
\[
\|e^{it\Delta} \phi\|_{X^{s_{n,k}}([0,T))} \le \|\phi\|_{H^{s_{n,k}}},
\]
and from \eqref{eq:non} we have
\begin{align*}
      \Big\|\mathcal{I}(|u|^{2k}u)\Big\|_{X^{s_{n,k}}([0,T))}
      \leq{}  c \|u\|^{2k+1}_{X^{s_{n,k}}([0,T))},
\end{align*}
and
  \begin{equation}\label{diff}
    \begin{split}
      &\Big\|\mathcal{I}(|u|^{2k}u-|v|^{2k}v)\Big\|_{X^{s_{n,k}}([0,T))}
      \\
      \leq{} & c
      (\|u\|^{2k}_{X^{s_{n,k}}([0,T))}+\|v\|^{2k}_{X^{s_{n,k}}([0,T))}) \|u-v\|_{X^{s_{n,k}}([0,T))},
    \end{split}
  \end{equation}
for all $0<T\leq 2\pi$ and $u,v\in X^{s_{n,k}}([0,T))$.

Now we assume that the initial data $\|\phi\|_{H^{s_{n,k}}}\leq \eps$ with $0<\eps\ll 1$ to be determined, and consider the compact set
  \begin{align*}
    D_{\delta}:=\{u \in X^{s_{n,k}}([0,2\pi))\cap C([0,2\pi);H^{s_{n,k}}(\T^n)):
    \|u\|_{X^{s_{n,k}}([0,2\pi))}\leq \delta\}.
  \end{align*}
where the parameters $\delta$ will be chosen latter. We are about to solve \eqref{int_eq} by the contraction mapping principle in $D_{\delta}$, for $\phi\in B_{\eps}$. We have
  \[
  \|e^{it\Delta} \phi + c\mathcal{I}(|u|^{2k}u)\|_{X^{s_{n,k}}([0,2\pi))}\leq \eps +c\delta^{2k+1}\leq \delta,
  \]
  by choosing
  \begin{equation}\label{eps}
    \delta=(4c)^{-\frac1{2k}} \text{ and } \eps=\delta/2.
  \end{equation}
By \eqref{diff} and \eqref{eps} we have
  \[
  \|\mathcal{I}(|u|^{2k}u - |v|^{2k}v)\|_{X^{s_{n,k}}([0,2\pi))}\leq \frac12 \|u-v\|_{X^{s_{n,k}}([0,2\pi))},
  \]
  which shows that the nonlinear map \eqref{int_eq} has a unique fixed
  point in $D_{\delta}$, thus we solve \eqref{eq:nls} with small initial data.
\end{proof}

\end{document}